\documentclass[11pt]{amsart}

\usepackage{verbatim, comment}
\usepackage{amssymb,amsthm,amsmath}
\usepackage[curve]{xypic}
\usepackage[breaklinks]{hyperref}
\usepackage{color}
\definecolor{darkblue}{rgb}{0, 0, .4}
\definecolor{grey}{rgb}{.7, .7, .7}
\hypersetup{
    colorlinks=true,
    linkcolor=darkblue,
    anchorcolor=darkblue,
    citecolor=darkblue,
    pagecolor=darkblue,
    urlcolor=darkblue,
    pdftitle={},
    pdfauthor={}
}

\newtheorem{theorem}{Theorem}[section]

\theoremstyle{definition}
\newtheorem{definition}[theorem]{Definition}
\newtheorem{example}[theorem]{Example}

\theoremstyle{remark}
\newtheorem{remark}[theorem]{Remark}
\numberwithin{equation}{section}
\theoremstyle{theorem}
\newtheorem{corollary}[theorem]{Corollary}

\newcommand{\lcm}[0]{\mathrm{lcm}}

\renewcommand{\b}[0]{\beta}

\renewcommand{\l}[0]{\ell}

\newcommand{\lf}[0]{L}
\newcommand{\nf}[0]{N}

\renewcommand{\mod}[0]{\ \mathrm{mod}\ }

\begin{document}

\title{Solitaire Mancala Games and the Chinese Remainder Theorem}

\begin{abstract}
Mancala is a generic name for a family of sowing games that are popular all over
the world.  There are many two-player mancala games in which a player may move
again if their move ends in their own store.  In this work, we study a simple
solitaire mancala game called Tchoukaillon that facilitates the analysis of
``sweep'' moves, in which all of the stones on a portion of the board can be
collected into the store.  We include a self-contained account of prior research
on Tchoukaillon, as well as a new description of all winning Tchoukaillon boards
with a given length.  We also prove an analogue of the Chinese Remainder Theorem
for Tchoukaillon boards, and give an algorithm to reconstruct a complete winning
Tchoukaillon board from partial information.  Finally, we propose a
graph-theoretic generalization of Tchoukaillon for further study.
\end{abstract}

\author{Brant Jones, Laura Taalman, and Anthony Tongen}
\address{Department of Mathematics and Statistics, MSC 1911, James Madison University, Harrisonburg, VA 22807}
\email{\texttt{[brant,taal,tongen]@math.jmu.edu}}


\date{\today}

\maketitle

\bigskip
\section{Introduction}\label{s:intro}

Mancala is a generic name for a family of ``sowing'' games that are popular all
over the world, particularly in Africa and parts of Asia.  Archaeological
evidence suggests that some games are at least 1,300 years old.  Russ
\cite{Russ} has collected a description of many of the variations.  A
commercially available version that is often played in America is called {\em
Kalah}.  It was invented and patented by William Julius Champion Jr. in the
1950's; see \cite{wiki_kalah} for rules and references.  Outside of the US, {\em
Oware} (also known as {\em Wari} or {\em Ayo}) is perhaps the most widespread
game in the Mancala family that is played on a $2 \times 6$ board.  These games
have also interested researchers in mathematics and machine learning
\cite{Donkers}.

In this work, we study a simple solitaire mancala game called {\em
Tchoukaillon} that facilitates the mathematical analysis of many of the other
game variations.  Tchoukaillon was introduced in 1977 \cite{delpop}, and is
derived from another Mancala variant called {\em Tchuka Ruma} that was first
described by Delannoy \cite{delannoy} in 1895.  See Campbell and Chavey
\cite{campbell-chavey} for a detailed mathematical analysis of Tchuka Ruma.

The Tchoukaillon board consists of a sequence of {\em bins} that can contain
stones, together with an additional store called the {\em Ruma}.  The goal of
Tchoukaillon is to move all of the stones originally on the board into the Ruma.
During each turn, a player may pick up all of the stones in a selected bin and
then sow them by depositing one stone in each succeeding bin towards the Ruma so
that the last stone is deposited in the Ruma.  For example,
Figure~\ref{f:ex_game} illustrates a valid Tchoukaillon game on three bins, with
the Ruma on the left side of the board.  On the other hand, it is not possible
to clear the board $(0; 1, 1, 0)$ under the rules of Tchoukaillon.

\begin{figure}[h]
    \[ (0; 0, 1, {\bf 3}) \rightarrow (1; {\bf 1}, 2, 0) \rightarrow (2; 0, {\bf 2}, 0) \rightarrow (3; {\bf 1}, 0, 0) \rightarrow (4; 0, 0, 0) \]
\caption{A Tchoukaillon game on three bins} \label{f:ex_game}
\end{figure}

There are many two-player mancala games in which a player may move again if
their sowing ends in their own store.  Therefore, the Tchoukaillon positions
represent ``sweep'' moves where all of the stones on a portion of the board can
be collected into the Ruma.  Understanding the Tchoukaillon positions yields
important strategic information about these games, particularly when other
features of the games become less important, such as during the opening or
endgame.

\subsection{Initial strategy and notation}\label{s:init}

Suppose we represent a Tchoukaillon board by a vector $(b_1, b_2, \ldots, b_k)$
where the Ruma lies to the left of $b_1$ and each $b_i$ indicates the number of
stones in the $i$th bin.  If we hope to eventually clear the board by making
moves so that the last stone of each sowing lands in the Ruma, then we must have
$b_i \leq i$ for all $i$; otherwise sowing $b_i$ would overshoot the Ruma and
we would never be able to clear these stones.

Let us say that a bin is {\em harvestable} if $b_i = i$.  At each move, we must
sow the stones from a harvestable bin so that the last stone of our sowing lands
in the Ruma.  If there is more than one harvestable bin, then we must choose the
one that is closest to the Ruma, for otherwise we will create a bin with $b_i >
i$ that can never be cleared.

These considerations tell us that if it is possible to clear a Tchoukaillon
board, then it must be done by moves that {\em sow the harvestable bin closest
to the Ruma}. 
In fact, we can invert this condition and ``unplay'' Tchoukaillon, starting from
the board that is initially empty.  Each such {\em unmove} has the following
form:
\begin{itemize}
\item Pick up a stone from Ruma, since that is where the last stone is always played.
\item Move away from the Ruma, picking up a stone from each nonempty bin.
\item When you arrive at an empty bin, drop all the collected stones.
\end{itemize}
For example, the first several Tchoukaillon boards are shown in
Figure~\ref{f:boards}.  Each board can be played for one move to sweep a single
stone into the Ruma, or unplayed for one move to add a new stone onto the board.
If we do not constrain the length of the board, this unplaying process can be
continued indefinitely.  Thus, we have shown that the game tree for Tchoukaillon
is actually a path, and that there exists a unique board having a given total
number of stones.

\begin{figure}[t]
\begin{tabular}{|c|c|| c c c c c c|}
  \hline
  $n$ & $\l$ & $b_1$ & $b_2$ & $b_3$ & $b_4$ & $b_5$ & $b_6$ \\ \hline
  0 & 0 & 0 & 0 & 0 & 0 & 0 & 0 \\
  1 & 1 & 1 & 0 & 0 & 0 & 0 & 0 \\
  2 & 2 & 0 & 2 & 0 & 0 & 0 & 0 \\
  3 & 2 & 1 & 2 & 0 & 0 & 0 & 0 \\
  4 & 3 & 0 & 1 & 3 & 0 & 0 & 0 \\
  5 & 3 & 1 & 1 & 3 & 0 & 0 & 0 \\
  6 & 4 & 0 & 0 & 2 & 4 & 0 & 0 \\
  7 & 4 & 1 & 0 & 2 & 4 & 0 & 0 \\
  8 & 4 & 0 & 2 & 2 & 4 & 0 & 0 \\
  9 & 4 & 1 & 2 & 2 & 4 & 0 & 0 \\
  10 & 5 & 0 & 1 & 1 & 3 & 5 & 0 \\
  11 & 5 & 1 & 1 & 1 & 3 & 5 & 0 \\
  12 & 6 & 0 & 0 & 0 & 2 & 4 & 6 \\
  13 & 6 & 1 & 0 & 0 & 2 & 4 & 6 \\
  14 & 6 & 0 & 2 & 0 & 2 & 4 & 6 \\
  15 & 6 & 1 & 2 & 0 & 2 & 4 & 6 \\
  16 & 6 & 0 & 1 & 3 & 2 & 4 & 6 \\
  17 & 6 & 1 & 1 & 3 & 2 & 4 & 6 \\
  \hline
\end{tabular}
\caption{Initial Tchoukaillon boards where $n$ is the total number of stones on
the board, $\l$ is the length of the board, and $b_i$ is the number of stones in bin $i$.} \label{f:boards}
\end{figure}

Throughout this paper, we will let $b(n)$ denote the vector $(b_1(n), b_2(n),
\ldots, b_i(n), \ldots)$, where $b_i(n)$ is the number of stones in the $i$th
bin of the unique winning Tchoukaillon board having $n$ total stones.  Here, we
number the bins beginning with $i = 1$ closest to the Ruma, increasing as we
move to the right.  More precisely, we have $b(0) = (0, 0, \ldots)$ and define
\begin{equation}\label{e:unplay}
  b_i(n+1) = \begin{cases}
    b_i(n) & \text{ if } i > p(n) \\
    i & \text{ if } i = p(n) \\
    b_i(n)-1 & \text{ if } i < p(n) \\
\end{cases}
\end{equation}
where $p(n) = \min\{j : b_j(n) = 0\}$ is the leftmost empty bin.  Note also that
\[ (p(n-1), p(n-2), \ldots, p(1)) \]
represents the sequence of bins that are played to actually win the board
$b(n)$.  Finally, we call 
\[ \lf(n) = \lf(b(n)) = \min \{i : \text{ $b_j(n) = 0$ whenever $j > i$ }\} \]
the {\em length} of the board $b(n)$.
Equivalently, we can define a sequence $\nf(\l) = \min \{n : \lf(n) = \l\}$.
These are the boards where the length increases, so the $\nf(\l)$ sequence begins
$1, 2, 4, 6, 10, 12, 18, \ldots$ according to the data in Figure~\ref{f:boards};
this is {\tt A002491} in \cite{sloanes}.

\subsection{Prior research}
Tchoukaillon was studied by Veronique Gautheron and introduced in 1977 by
Deledicq and Popova \cite{delpop}.  The authors proved the results about
unplaying and the uniqueness of boards with a given number of stones that we
related in Section~\ref{s:init}.  They also posed two natural questions for
further research:
\begin{enumerate}
    \item[(1)]  Can the $n$th winning board $b(n)$ be found without iterating
        through all of the prior boards?
    \item[(2)]  What is the function $\lf(n)$ (asymptotically or explicitly)?
\end{enumerate}

The next substantial results appear to be Betten's 1988 paper \cite{Betten}.
He did not cite \cite{delpop} and appears to have been motivated primarily by
the Sloane's encyclopedia entry for the sequence $\nf(\l)$.  Betten answered
question (1) and related the sequence $\nf(\l)$ to a generalized sieve of
Eratosthenes that had been studied in the late 1950's by Erd\"os, Jabotinsky,
and David \cite{Erdos,David}.  This enabled him to give the asymptotic formula 
\[ \nf(\l) = \frac{\l^2}{\pi} + O(\l^{4/3}) \]
which provides an answer for question (2).  

Betten also observed and proved the important result that the sequence
$\{b_i(n)\}_{n=0}^{\infty}$ obtained from the $i$th bin is periodic with period
$\lcm(2, 3, \ldots, i+1)$.  For example, the data we have displayed in
Figure~\ref{f:boards} is seen to be $2$-periodic in column $1$, $6$-periodic in
column $2$, and $12$-periodic in column $3$.

In 1995, Broline and Loeb \cite{Broline} gave a detailed mathematical analysis
of Tchoukaillon, motivated in part by a paper \cite{campbell-chavey} of
Campbell and Chavey on the related game Tchuka Ruma.  It seems that they were
not aware of Betten's paper.  They observed the periodicity, and also made the
connection to the sieving work from the 1950's.  In fact, they were able to
strengthen the asymptotic formula for $\nf(\l)$ to
\[ \nf(\l) = \frac{\l^2}{\pi} + O(\l) \]
proving a conjecture from Erd\"os--Jabotinsky \cite{Erdos}.  This yields a
method to approximate $\pi$ using Tchoukaillon.

However, they did not answer question (1).  The AMS Math Review by Richard
Nowakowski of \cite{Broline} states ``Given [the total number of stones], no way
is known at this time to quickly determine the winning arrangement,'' so this
seems to have been viewed as an open problem in the combinatorial games
community.

\subsection{Outline}

This paper gives a self-contained account of the results mentioned above and
provides some ideas for further research.  In particular, we give a connection
with the Chinese Remainder Theorem that appears to be new.

We begin in Section~\ref{s:nonit} with a simple formula answering question (1).
Once this result is in hand, it is easy to obtain the periodicity result.  In
Section~\ref{s:length_boards}, we develop a new dual answer to question (1),
and show how to construct all boards of a given length.  This leads to a
formula for $\nf(\l)$, answering question (2) explicitly.  As we have
indicated, these results were known to earlier researchers, but our proofs are
much more straightforward.  In Section~\ref{s:crt}, we consider boards in which
some subset of the $b_i$ have been specified and ask when these can be
completed to a winning Tchoukaillon board.  This leads to an analogue of the
Chinese Remainder Theorem.  Finally, in Section~\ref{s:conclusions}, we propose
a generalization of Tchoukaillon that can be played on any directed graph and
consider some directions for future research.

\bigskip
\section{Non-iterative board construction}\label{s:nonit}

We find that $b(n)$ can be characterized as the unique sequence so that the sum
of the stones in the first $i$ bins is always equivalent to $n$ mod $(i+1)$.
This result also appears in \cite{Betten}.

\begin{theorem}\label{t:nonit}
Fix $n \geq 0$.  The $b_i(n)$ satisfy
\begin{equation}\label{e:modleft}
\sum_{j=1}^{i} b_j(n) \equiv n \mod (i+1), \text{ for each $i
\geq 1$. }
\end{equation}

In fact, this uniquely determines the $b_i(n)$, given $n$, as
\[ b_i(n) = \left(n - \sum_{j=1}^{i-1} b_j(n)\right) \mod (i+1), \]
normalized so that $0 \leq b_i(n) \leq i$.
\end{theorem}

\begin{example}
Suppose we wish to find the unique winning board with $n=15$ stones, without
unplaying from the trivial board.  We can use the formula given in
Theorem~\ref{t:nonit} to obtain
\[ b_1(15) = 15 \mod 2 = 1, \]
\[ b_2(15) = (15 - b_1(15)) \mod 3 = 14 \mod 3 = 2, \]
\[ b_3(15) = (15 - b_1(15) - b_2(15)) \mod 4 = 12 \mod 4 = 0, \]
\[ b_4(15) = (15 - b_1(15) - b_2(15) - b_3(15)) \mod 5 = 12 \mod 5 = 2, \]
\[ b_5(15) = (15 - b_1(15) - b_2(15) - b_3(15) - b_4(15)) \mod 6 = 10 \mod 6 = 4, \]
\[ b_6(15) = (15 - b_1(15) - b_2(15) - b_3(15) - b_4(15) - b_5(15)) \mod 7 = 6 \mod 7 = 6, \]
and $b_i(15) = 0$ for all $i \geq 7$.  Hence, the board is $(1,2,0,2,4,6)$.
\end{example}

We can view the formula from Theorem~\ref{t:nonit} as a reverse sowing game.
We begin with $n$ stones in our hand, and sow $b_i$ stones into the $i$th bin
in such a way that the number of stones remaining in our hand is divisible by
$i+1$ and as large as possible, at each step.  When $n = 15$ for example, we
sow $1$ stone into the first bin in order to leave $14$ stones in our hand
(which is divisible by $2$); next, we sow $2$ stones into the second bin to
leave $12$ stones in our hand (which is divisible by $3$); etc.

\begin{proof}
Since the $b_i(n)$ must satisfy $0 \leq b_i(n) \leq i$ in order to form a
winnable board that never places stones beyond the Ruma, the first and second
statements of the theorem are equivalent.

We work by induction to prove the first statement, so suppose that this
statement is true for $b(n)$ and consider the board $b(n+1)$.  We obtain
$b(n+1)$ from $b(n)$ by unplaying into some bin, say $p$.  Here, $p$ is the
leftmost empty bin and Equation~(\ref{e:unplay}) expresses $b(n+1)$ in terms of
$b(n)$ as a piecewise function.

Unplaying adds exactly one stone to the board and fixes the bins to the right
of the $p$th bin.  Hence, for each $i \geq p$ we have
\[ \sum_{j=1}^i b_j(n+1) = \left( \sum_{j=1}^i b_j(n) \right) + 1. \]
This quantity is equivalent to $(n+1) \mod (i+1)$, by induction.

For $i < p$, we have
\[ \sum_{j=1}^i b_j(n+1) = \sum_{j=1}^i (b_j(n) - 1) = \left(\sum_{j=1}^i b_j(n)\right) - i. \]
This quantity is equivalent to $(n+1) \mod (i+1)$, by induction.
\end{proof}

This result allows us to give a simple proof that the boards are periodic in
$n$, a fact that was also observed in \cite{Betten,Broline} using different
reasoning.

\begin{corollary}\label{c:periodicity}
Fix $i > 0$.  For all $n, k \geq 0$, we have
\[ (b_1(n), b_2(n), \ldots, b_i(n)) = (b_1(n+k), b_2(n+k), \ldots, b_i(n+k)) \]
if and only if $\lcm(2, 3, \ldots, i+1)$ divides $k$.
\end{corollary}
\begin{proof}
Fix $i$ and let $m = \lcm(2, 3, \ldots, i+1)$.
Since $m$ is a multiple of $2$, we have that $b_1(n) = b_1(n+m) \mod 2$.
Then since $m$ is a multiple of $3$, we also have that
\[ b_2(n+m) = ( (n+m) - b_1(n+m) ) \mod 3 = ( n - b_1(n) ) \mod 3 = b_2(n) \]
by Theorem~\ref{t:nonit}.  Continuing in this fashion, we have that $b_j(n) =
b_j(n+m)$ for all $1 \leq j \leq i$.

Conversely, if $b_j(n) = b_j(n+k)$ for all $1 \leq j \leq i$, then their partial
sums are also equal, so we must have that $n \equiv n+k \mod (j+1)$ for each $1
\leq j \leq i$ by Theorem~\ref{t:nonit}.  Hence, $k$ is a multiple of $2, 3,
\ldots, i+1$.  Therefore, $\lcm(2, 3, \ldots, i+1)$ is the minimal period of the
sequences $(b_1(n), b_2(n), \ldots, b_i(n))$.
\end{proof}

\bigskip
\section{Boards with prescribed length}\label{s:length_boards}

In this section, we explore a dual setting to that of Section~\ref{s:nonit}.
Rather than developing the board from the bins closest to the Ruma and working
to the right, we instead fix the length of the board and work left, from the
furthest bin towards the Ruma.  In contrast to the results of the previous
section where we found an expression for the unique winning board with a given
total number of stones, there are many winning Tchoukaillon boards of a given
length.

We begin with a dual form of Theorem~\ref{t:nonit}.

\begin{corollary}\label{c:niright}
Fix $n \geq 0$.  The $b_i(n)$ satisfy
\begin{equation}\label{e:modright}
\sum_{j=i}^{\infty} b_j(n) \equiv 0 \mod i, \text{ for each $i \geq 1$. }
\end{equation}
\end{corollary}
\begin{proof}
The sum on the left is finite since $b_j(n) = 0$ for all $j > \lf(n)$.
Moreover, $\sum_{j=1}^{\infty} b_j(n) = n$ by definition.
By Theorem~\ref{t:nonit}, we may subtract Equation~(\ref{e:modleft}) from the
equation 
\[ \sum_{j=1}^{\infty} b_j(n) \equiv n \mod (i+1) \]
obtaining $\sum_{j=1}^{\infty} b_j(n) - \sum_{j=1}^{i} b_j(n) \equiv n-n \mod
(i+1)$, for each $i \geq 1$, which yields the result.
\end{proof}

We say that a sequence $(b_1, b_2, \ldots, b_k)$ of positive integers {\em
represents a winning Tchoukaillon board} if there exists $n$ such that
\[ b_i(n) = \begin{cases} b_i & \text{ if $i \leq k$ } \\
    0 & \text{ if $i > k$. } \\
\end{cases} \]

\begin{theorem}\label{T:uppersums}
Fix $k > 0$.  A sequence of positive integers $(b_1,b_2,\ldots,b_k)$ represents
a winning Tchoukaillon board if and only if for all $1 \leq i \leq k$ we have $b_i \leq i$ and
\begin{equation}\label{e:modrightfin}
\sum_{j=i}^k b_j \equiv 0 \mod i.
\end{equation}
\vspace{-.5\baselineskip}
\end{theorem}
\begin{proof}
Suppose we have a sequence $(b_1, b_2, \ldots, b_k)$ as in the statement and let
$n = \sum_{j=1}^k b_j$.  By subtracting Equation~(\ref{e:modrightfin}) from
\[ \sum_{j=1}^k b_j \equiv n \mod i, \]
we find that $(b_1, b_2, \ldots, b_k)$ satisfies Equation~(\ref{e:modleft}) so
agrees with $b(n)$ by Theorem~\ref{t:nonit}.

Conversely, any $b(n)$ satisfies the conditions in the statement by
Corollary~\ref{c:niright}.
\end{proof}

\begin{example}
The board $(b_1,b_2,b_3,b_4,b_5,b_6)=(1,2,0,2,4,6)$ has entries that satisfy the following congruences:
\begin{align*}
b_6 &= 6 & (6 \equiv 0 \mod 6) \\
b_5+b_6 &= 4+6 = 10 & (10 \equiv 0 \mod 5) \\
b_4+b_5+b_6 &= 2+4+6 = 12 & (12 \equiv 0 \mod 4) \\
b_3+b_4+b_5+b_6 &= 0+2+4+6 = 12 & (12 \equiv 0 \mod 3) \\
b_2+b_3+b_4+b_5+b_6 &= 2+0+2+4+6 = 14 & (14 \equiv 0 \mod 2) \\
b_1+b_2+b_3+b_4+b_5+b_6 &= 1+2+0+2+4+6 = 15 & (15 \equiv 1 \mod 1)
\end{align*}
\end{example}

Using the converse of Theorem~\ref{T:uppersums} we can construct {\em all} of the
winning Tchoukaillon boards of any given length.  For example, if
$(b_1,b_2,b_3,b_4,b_5,b_6)$ is a winning Tchoukaillon board of length $k=6$,
then we must have $b_6=6$.  By Theorem~\ref{T:uppersums} we must have $b_5 \leq
5$ and $b_5+b_6$ a multiple of 5, and thus $b_5=4$.  Similarly, we must have
$b_4 \leq 4$ and $b_4+b_5+b_6$ a multiple of $4$, which gives $b_4=2$.
Repeating this process, we have $b_3 \leq 3$ and $b_3+b_4+b_5+b_6$ a multiple of
$3$.  Since $b_3+b_4+b_5+b_6=12$ is already a multiple of 3 there are two
options; either $b_3=0$ or $b_3=3$.  We repeat this process until we find $b_1$
and have all possible sequences of length $k$ that satisfy Theorem~\ref{T:uppersums}.
The six winning Tchoukaillon boards of length six are shown in
Figure~\ref{f:boards}.

Given a positive integer $n$ there is exactly one winning Tchoukaillon board
with $n$ stones, and that board has a unique length $\l$.  Turning this around,
for a given positive integer $\l$ there can be many winning Tchoukaillon boards
with length $\l$, each with a different number of stones $n$.  However there
will always be a unique minimum number of stones $n = \nf(\l)$ that is possible in
a length $\l$ winning Tchoukaillon board.  Question (2) of Deledicq and Popova
asks for the precise relationship between $n$ and $\l$.

\begin{example}
In this example, we show how to determine the smallest number of stones
required to build a board with a given length.  The winning Tchoukaillon boards
of length six are represented in the left side of Figure~\ref{f:board_trees}.
The right side of Figure~\ref{f:board_trees} shows the corresponding upper
partial sums.

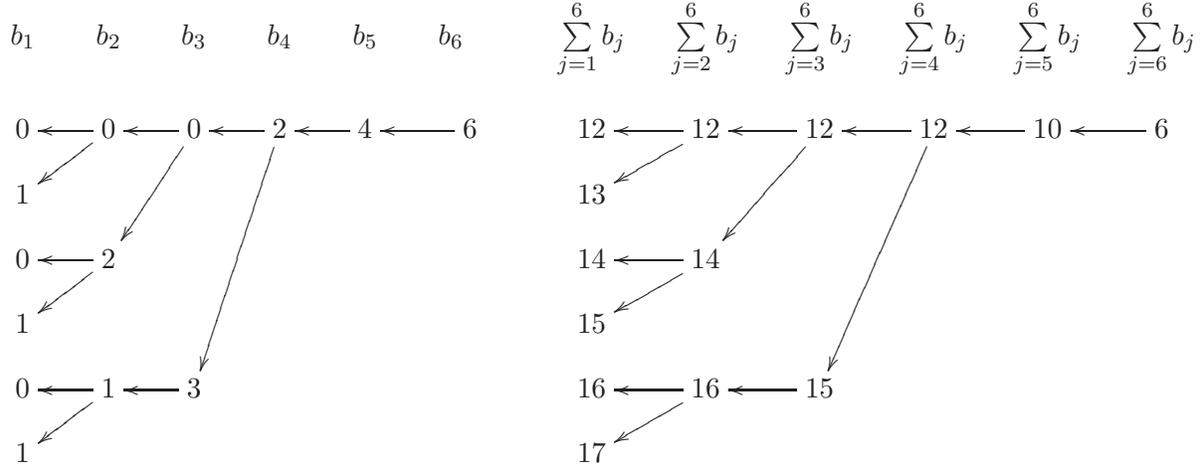
\begin{figure}[ht]
$$\xymatrix@C=1.5em@R=1em{
b_1 & b_2 & b_3 & b_4 & b_5 & b_6\phantom{\sum\limits_{j=1}^6} \\
0 \ar@{<-}[r]
   & 0 \ar@{<-}[r]
   & 0 \ar@{<-}[r]
   & 2 \ar@{<-}[r]
   & 4 \ar@{<-}[r]
   & 6 \\
1 \ar@{<-}[ur] \\
0 \ar@{<-}[r]
   & 2 \ar@{<-}[uur] \\
1 \ar@{<-}[ur] \\
0 \ar@{<-}[r]
   & 1 \ar@{<-}[r]
   & 3 \ar@{<-}[uuuur] \\
1 \ar@{<-}[ur]
}
\hspace{.2in}
\xymatrix@C=1em@R=1em{
\sum\limits_{j=1}^6 b_j
   & \sum\limits_{j=2}^6 b_j
   & \sum\limits_{j=3}^6 b_j
   & \sum\limits_{j=4}^6 b_j
   & \sum\limits_{j=5}^6 b_j
   & \sum\limits_{j=6}^6 b_j \\
12 \ar@{<-}[r]
   & 12 \ar@{<-}[r]
   & 12 \ar@{<-}[r]
   & 12 \ar@{<-}[r]
   & 10 \ar@{<-}[r]
   & 6 \\
13 \ar@{<-}[ur] \\
14 \ar@{<-}[r]
   & 14 \ar@{<-}[uur] \\
15 \ar@{<-}[ur] \\
16 \ar@{<-}[r]
   & 16 \ar@{<-}[r]
   & 15 \ar@{<-}[uuuur] \\
17 \ar@{<-}[ur]
}$$
\caption{Tchoukaillon boards of length six.} \label{f:board_trees}
\end{figure}

The process followed to obtain the top row in the right side of
Figure~\ref{f:board_trees} is equivalent to the following.  From right to left,
starting with $\l=6$, increase if necessary to obtain the next multiple of $5$,
which is $10$.  Then increase to the next multiple of $4$, which is $12$.  This
is already a multiple of $3$ so we stay at $12$, and again we stay at $12$ for
multiples of $2$ and $1$.  Notice that we have just shown that $12$ is the
smallest possible number of stones for a Tchoukaillon board of length $\l=6$.
\end{example}

In general, we have the following formula.

\begin{theorem}
The minimum number of stones $\nf(\l)$ for a winning Tchoukaillon board of length $\l$
is given by the formula
$$\nf(\l)=\tfrac{2}{1}\lceil\tfrac{3}{2}\lceil\cdots\lceil\tfrac{\l-1}{\l-2}\lceil\tfrac{\l}{\l-1}\rceil\rceil\cdots\rceil\rceil.$$
\end{theorem}

\begin{proof}
We will use the facts that given any nonnegative integers $r$, $s$ and $k$, the
next highest multiple of $k$ weakly greater than $r$ is $k \lceil \frac{r}{k}
\rceil$, and $r < s$ implies $k \lceil \tfrac{r}{k} \rceil \leq k \lceil
\tfrac{s}{k} \rceil$.

We apply Theorem~\ref{T:uppersums} to construct a board of length $\l$ having the
fewest number of stones.  By Theorem~\ref{T:uppersums}, we begin with $b_\l = \l$
and then work to the left, choosing $b_i$ so that $b_i + \sum_{j=i+1}^\l b_j$ is
the next highest multiple of $i$ that is weakly greater than $\sum_{j=i+1}^\l
b_j$.  If we assume that  
\[ \sum_{j=i+1}^\l b_j = (i+1)\lceil\tfrac{i+2}{i+1}\lceil\cdots\lceil\tfrac{\l-1}{\l-2}\lceil\tfrac{\l}{\l-1}\rceil\rceil\cdots\rceil\rceil \] 
and that $\sum_{j=i+1}^\l b_j$ is minimal among all boards of length $\l$,
then our choice of $b_i$ ensures that
\[ \sum_{j=i}^\l b_j =
(i)\lceil\tfrac{i+1}{i}\lceil\cdots\lceil\tfrac{\l-1}{\l-2}\lceil\tfrac{\l}{\l-1}\rceil\rceil\cdots\rceil\rceil
\]
continues to be minimal, so the construction proceeds by induction.
\end{proof}

We can also prove a very rough asymptotic estimate for these numbers
combinatorially, using our knowledge of Tchoukaillon.

\begin{theorem}
The function $\nf(\l)$ is $O(\l^2)$.  In particular,
\[ \frac{\l^2}{4} + O(\l) \leq \nf(\l) \leq \frac{\l^2}{2} + O(\l). \]
\end{theorem}
\begin{proof}
Suppose $n$ is the minimum number of stones in a winning Tchoukaillon board of
length $\l$.  At each step we have $b_i(n) \leq i$, so
\[ \nf(\l) = \sum_{i=1}^{\l} b_i(n) \leq \sum_{i}^{\l} i = { {\l+1} \choose 2} = \frac{\l^2}{2} + O(\l). \]

On the other hand, Theorem~\ref{T:uppersums} shows that we obtain $\nf(\l)$ from
$\l$ by iteratively adding $b_i$ stones to $\l + \sum_{j=i+1}^{\l} b_j$ so that the
result is a multiple of $i$, for each $i$ from $\l$ down to $2$.
Since $\l + (\l - 2) = 2 (\l - 1)$, we have $b_{\l-1} = \l - 2$.  Since $\l + (\l - 2)
+ (\l - 4) = 3 (\l - 2)$, we have $b_{\l-2} = \l - 4$.  Continuing in this way, we
have $b_{\l - i} = \l - 2i$ for each $i$ from $0, \ldots, \lfloor \frac{\l}{2}
\rfloor$.  Hence,
\[ \nf(\l) = \sum_{i=1}^{\l} b_i(n) \geq 
\sum_{i = 0}^{\lfloor \l/2 \rfloor} b_{\l-i}(n)
= \sum_{i = 0}^{\lfloor \l / 2 \rfloor} (\l - 2i)
\geq \l \left(\frac{\l}{2}\right) - 2 { {\frac{\l}{2}+1} \choose 2 } \]
\[ = \frac{\l}{2} \left(\frac{\l}{2} - 1\right)
= \frac{\l^2}{4} + O(\l). \]
\end{proof}

Interestingly, the true asymptotic coefficient of $\l^2$ in $\nf(\l)$ is $1 / \pi$,
as shown in \cite{Betten,Broline}.  As we mentioned in Section~\ref{s:init}, one
proof of this result employs properties of a sequence of integers generated by
a sieving process that was studied in the late 1950's by Erd\"os, Jabotinsky
and David \cite{David,Erdos}, before Tchoukaillon was invented.

Here, we define the sieve in terms of Tchoukaillon and show that it is the same
sieve that was studied by Erd\"os et al.  Recall that $p(n-1)$ is the unique
minimal $i$ such that $b_i(n) = i$.  That is, $p(n-1)$ is the bin that is played
to win the board with $n$ stones.  We consider a sieve process in which we begin
with all of the integers
\[ S_i^{(1)} := i \]
and then let $S_i^{(k)}$ be the $i$th integer remaining after we have removed
all integers $n$ such that $p(n-1) < k$.

The first few sequences from this process are shown below.  For example, $p(n-1)
= 1$ if and only if $n$ is odd, so $\{S_i^{(2)}\}_{i \geq 1}$ consists of all
the even numbers.  Also, the sequence $\nf(\l)$ is encoded as $S_1^{(m)}$.

\[ \begin{matrix}
    S_1^{(2)} & S_2^{(2)} & S_3^{(2)} & S_4^{(2)} & S_5^{(2)} & S_6^{(2)} & S_7^{(2)} & S_8^{(2)} & S_9^{(2)} & \cdots \\
    2 & 4 & 6 & 8 & 10 & 12 & 14 & 16 & 18 & \cdots \\
\end{matrix} \]

\[ \begin{matrix}
    S_1^{(3)} & S_2^{(3)} & S_3^{(3)} & S_4^{(3)} & S_5^{(3)} & S_6^{(3)} & S_7^{(3)} & S_8^{(3)} & S_9^{(3)} & \cdots \\
    4 & 6 & 10 & 12 & 16 & 18 & 22 & 24 & 28 & \cdots \\
\end{matrix} \]

\[ \begin{matrix}
    S_1^{(4)} & S_2^{(4)} & S_3^{(4)} & S_4^{(4)} & S_5^{(4)} & S_6^{(4)} & S_7^{(4)} & S_8^{(4)} & S_9^{(4)} & \cdots \\
    6 & 10 & 12 & 18 & 22 & 24 & 30 & 34 & 36 & \cdots \\
\end{matrix} \]

\[ \begin{matrix}
    S_1^{(5)} & S_2^{(5)} & S_3^{(5)} & S_4^{(5)} & S_5^{(5)} & S_6^{(5)} & S_7^{(5)} & S_8^{(5)} & S_9^{(5)}  & \cdots \\
    10 & 12 & 18 & 22 & 30 & 34 & 36 & 42 & 48 & \cdots \\
\end{matrix} \]

We now show that this sieve which we have defined in terms of Tchoukaillon is
the same sieve studied in \cite{David,Erdos}.  Betten \cite{Betten} seems to
have been the first to observe this connection.

\begin{theorem}
We have that $\{S_i^{(k+1)}\}_{i \geq 1}$ is always obtained from
$\{S_i^{(k)}\}_{i \geq 1}$ by removing the elements $\{S_{j(k+1)+1}^{(k)}\}_{j
\geq 0}$ and re-indexing.  
\end{theorem}
\begin{proof}
We work by induction.  The claim is true for $S_i^{(2)}$, so suppose that we
have verified the claim for $S_i^{(k)}$.  Recall that each board $b(n)$ can be
built from the board $b(n-1)$ by unplaying.

Every board $b(n)$ for $n \notin \{S_i^{(k)}\}_{i \geq 1}$ has its initial
play in some bin closer to the Ruma than bin $k$.  The board $b(S_1^{(k)})$ is
the minimal board with this property so bin $k$ is empty until we unplay into
it, depositing $k$ stones, hence $k$ is the first bin played on $b(S_1^{(k)})$.
Therefore, $S_1^{(k)}$ will not appear in $\{S_i^{(k+1)}\}_{i \geq 1}$.

We follow the unplay algorithm to go from board $b(S_1^{(k)})$ to subsequent
boards $b(S_{i}^{(k)})$ in the sequence $\{S_{i}^{(k)}\}_{i \geq 1}$.  When the
unmove deposits stones in a bin that is closer to the Ruma than bin $k$, there
is no effect on the number of stones in bin $k$, and this corresponds to one of
the boards that we already removed in the sieve process.  Hence, we can ignore
these moves.

Otherwise, our unmove removes a stone from bin $k$ and deposits stones in some
bin further from the Ruma than bin $k$.  Each of these boards appear in the
sequence $\{S_i^{(k)}\}_{i \geq 1}$ by induction.  This type of unmove is repeated $k$ times in
total until all of the stones have been removed from bin $k$.  The subsequent
unmove that affects bin $k$ deposits $k$ stones into bin $k$, and so
$S_{1+(k+1)}^{(k)}$ does not appear in $\{S_i^{(k+1)}\}_{i \geq 1}$.  Moreover,
the process begins again and so we have that $S_{1+j(k+1)}^{(k)}$ does not
appear in $\{S_i^{(k+1)}\}_{i \geq 1}$ for all $j \geq 0$.
\end{proof}

We remark that although $S_i^{(2)}$ is obtained by removing all the integers of
the form $2j+1$ and $S_i^{(3)}$ is obtained by removing all integers of the form
$6j+2$, the sieving process does not in general remove arithmetic sequences.

\bigskip
\section{Partial board reconstruction and the Chinese Remainder Theorem}\label{s:crt}

We now turn to the problem of reconstructing a Tchoukaillon board given only
partial information.  In order to simplify formulas involving remainders, we
will index the Tchoukaillon bins starting from $2$, rather than from $1$ in
this section.  That is, we define $\b_i(n) = b_{i-1}(n)$ and apply all of our
prior results to $\b(n)$.

Theorem~\ref{t:nonit} says that the partial sums of the bin sequence in
Tchoukaillon form valid residue classes of a single integer.  This allows us to
connect the combinatorics of the Tchoukaillon boards to sequences obtained from
the Chinese Remainder Theorem.  In order to facilitate an analogy between
these,  we construct an infinite {\em remainder board} 
\[ c(n) = (c_2(n), c_3(n), \ldots) \]
where $c_i(n) = n \mod i$ for each nonnegative integer $n$, normalized so that
$0 \leq c_i(n) < i$.  We would also like to define a sequence that agrees with
$c_i(n)$ mod $i$ but is increasing.

\begin{definition}
Let $\widetilde{c}_2(n) = c_2(n)$ and define
$\widetilde{c}_i(n)$ to be the next integer weakly greater than
$\widetilde{c}_{i-1}(n)$ that is equivalent to $c_i(n)$ mod $i$.
We call $\widetilde{c}(n) = (\widetilde{c}_2(n), \widetilde{c}_3(n), \ldots)$
an {\em increasing remainder board}.
\end{definition}

\begin{figure}[ht]
\begin{tabular}{|c|| c c c c c c || c c c c c c |}
  \hline
$n$ & $c_2$ & $c_3$ & $c_4$ & $c_5$ & $c_6$ & $c_7$ & 
$\widetilde{c}_2$ & $\widetilde{c}_3$ & $\widetilde{c}_4$ & $\widetilde{c}_5$ &
$\widetilde{c}_6$ & $\widetilde{c}_7$ \\
  \hline
0 & 0 & 0 & 0 & 0 & 0 & 0 &    0 & 0 & 0 & 0 & 0 & 0 \\
1 & 1 & 1 & 1 & 1 & 1 & 1 &    1 & 1 & 1 & 1 & 1 & 1 \\
2 & 0 & 2 & 2 & 2 & 2 & 2 &    0 & 2 & 2 & 2 & 2 & 2 \\
3 & 1 & 0 & 3 & 3 & 3 & 3 &    1 & 3 & 3 & 3 & 3 & 3 \\
4 & 0 & 1 & 0 & 4 & 4 & 4 &    0 & 1 & 4 & 4 & 4 & 4 \\
5 & 1 & 2 & 1 & 0 & 5 & 5 &    1 & 2 & 5 & 5 & 5 & 5 \\
6 & 0 & 0 & 2 & 1 & 0 & 6 &    0 & 0 & 2 & 6 & 6 & 6 \\
7 & 1 & 1 & 3 & 2 & 1 & 0 &    1 & 1 & 3 & 7 & 7 & 7 \\
8 & 0 & 2 & 0 & 3 & 2 & 1 &    0 & 2 & 4 & 8 & 8 & 8 \\
9 & 1 & 0 & 1 & 4 & 3 & 2 &    1 & 3 & 5 & 9 & 9 & 9 \\
10 & 0 & 1 & 2 & 0 & 4 & 3 &    0 & 1 & 2 & 5 & 10 & 10 \\
11 & 1 & 2 & 3 & 1 & 5 & 4 &    1 & 2 & 3 & 6 & 11 & 11 \\
12 & 0 & 0 & 0 & 2 & 0 & 5 &    0 & 0 & 0 & 2 & 6 & 12 \\
13 & 1 & 1 & 1 & 3 & 1 & 6 &    1 & 1 & 1 & 3 & 7 & 13 \\
14 & 0 & 2 & 2 & 4 & 2 & 0 &    0 & 2 & 2 & 4 & 8 & 14 \\
15 & 1 & 0 & 3 & 0 & 3 & 1 &    1 & 3 & 3 & 5 & 9 & 15 \\
16 & 0 & 1 & 0 & 1 & 4 & 2 &    0 & 1 & 4 & 6 & 10 & 16 \\
17 & 1 & 2 & 1 & 2 & 5 & 3 &    1 & 2 & 5 & 7 & 11 & 17 \\
  \hline
\end{tabular}
\caption{A list of the first remainder boards}\label{f:crboards}
\end{figure}

Initial subsequences of the first several remainder boards are shown in
Figure~\ref{f:crboards}.  We observe that the remainder boards are in natural
bijection with the Tchoukaillon boards.  In fact, the Tchoukaillon boards can be
viewed as a finite difference or ``derivative'' of the increasing remainder
boards.

\begin{theorem}\label{t:ctob}
For all $n$, we have 
\[ \widetilde{c}_i(n) = \sum_{j=2}^i \b_j(n), \hspace{0.2in} \text{ and } \hspace{0.2in} \b_i(n) = \widetilde{c}_{i}(n) - \widetilde{c}_{i-1}(n). \]
\end{theorem}
\begin{proof}
Fix $n$ and let $\widetilde{\b}_i$ be defined by $\widetilde{c}_{i}(n) -
\widetilde{c}_{i-1}(n)$.  Then $0 \leq \widetilde{\b}_i < i$ and
$\widetilde{c}_1(n) = 0$ so we have $\sum_{j=2}^i \widetilde{\b}_i =
\widetilde{c}_{i}(n) - \widetilde{c}_1(n) = \widetilde{c}_{i}(n)$.  This is
$\widetilde{c}_i(n) \equiv c_{i}(n) \equiv n \mod i$, so Theorem~\ref{t:nonit}
implies that $\widetilde{\b}_i(n) = \b_i(n)$.
\end{proof}

For example, if $n = 29$, then $c(29) = (1, 2, 1, 4, 5, 1, 5, 2, 9, 7, \ldots)$, so
\[ \widetilde{c}(29) = (1, 2, 5, 9, 11, 15, 21, 29, 29, 29, \ldots). \]
The corresponding Tchoukaillon board is $\b(29) = (1, 1, 3, 4, 2, 4, 6, 8, 0, 0, \ldots)$.

Many properties of Tchoukaillon boards are reflected in the remainder boards
and increasing remainder boards.  For example, the $\widetilde{c}(n)$ exhibit
the same $\lcm(2, \ldots, i)$-periodicity as the $\b(n)$, while the $c(n)$ have
a stronger form of periodicity in which the entries of each column repeat as a
block.  Also, the $\widetilde{c}(n)$ eventually stabilize at column $\lf(n)$,
so the explicit and asymptotic formulas for $\nf(\l)$ that have been developed for
Tchoukaillon apply equally well to $\widetilde{c}(n)$.

We say that a sequence $(m_{i_1}, m_{i_2}, \ldots, m_{i_k})$ of
integers {\em agrees with a winning Tchoukaillon board} if there exists
$n$ such that
\[ \b_{i_j}(n) = m_{i_j} \text{ for all $j \in \{1, \ldots, k\}$. } \]
Similarly, we say that the sequence {\em agrees with a remainder board} if there
exists $n$ such that 
\[ c_{i_j}(n) = m_{i_j} \text{ for all $j \in \{1, \ldots, k\}$. } \]

The problem of determining when a sequence agrees with a remainder board is
solved by the Chinese Remainder Theorem.

\begin{theorem}{\bf (Chinese Remainder Theorem)}\label{t:crt}
Fix a sequence $(m_{i_1}, m_{i_2}, \ldots, m_{i_k})$ of integers with $0 \leq
m_{i_j} < i_j$ for all $j$.  The sequence agrees with a remainder board if and
only if 
\[ m_{i_p} \equiv m_{i_q} \mod \gcd(i_p, i_q) \text{ for all $p \neq q$. } \]
Moreover, if the sequence agrees with the remainder board $c(n)$ then it also
agrees with the remainder boards $c(n + r (\lcm(i_1, i_2, \ldots, i_k)))$
for all $r \in \mathbb{Z}$.
\end{theorem}

When a sequence agrees with a remainder board $c(n)$, there are constructive
algorithms to produce $n$ from the sequence.  In any case, the last part of the
theorem shows that the minimal $n$ is less than $\lcm(i_1, i_2, \ldots, i_k)$,
so can be found in finitely many steps.

We now turn to the analogous question for Tchoukaillon boards.

\begin{example}\label{e:notrelprime}
Consider whether there exists a Tchoukaillon board having $\b_5(n) = 1$ and
$\b_6(n) = 2$.  We may let $x$ denote $\b_2(n) + \b_3(n) + \b_4(n)$, and then we
have
\[ n \equiv x \mod 4 \]
\[ n \equiv x+1+2 \mod 6 \]
by Theorem~\ref{t:nonit}.  These equations force $n$ to have opposite parity,
so no such $n$ exists.

Conversely, it is straightforward to verify by computer that for any fixed $0
\leq x \leq 3$ and $0 \leq y \leq 6$, there is always a Tchoukaillon board
having $\b_4(n) = x$ and $\b_7(n) = y$.  The parameters $(x, y) = (2, 5)$ give the
smallest board with $18$ stones, while $(x,y) = (1,0)$ requires 214 stones.
\end{example}

In general, we have the following analogue of the Chinese Remainder Theorem for
Tchoukaillon boards.  

\begin{theorem}\label{t:tcrt}
Fix a sequence $(m_2, m_3, \ldots, m_k)$ of integers with $0 \leq m_i < i$ for
all $i$.  The sequence agrees with a Tchoukaillon board $\b(n)$ if and only if 
\begin{equation}\label{e:tcrt}
m_i + m_{i-1} + \cdots + m_{i-d+1} \equiv 0 \mod d
\end{equation}
for each nontrivial proper divisor $d$ of $i$ that is a power of a prime.
Moreover, if the sequence agrees with the Tchoukaillon board $b(n)$ then it also
agrees with the Tchoukaillon board $b(n + r (\lcm(2, \ldots, k)))$ for all
$r \in \mathbb{Z}$.
\end{theorem}

This result generalizes Theorem~\ref{T:uppersums} to the situation in which $k$
is not necessarily $\lf(n)$.  The complete list of conditions on the $m_i$ for
$k \leq 12$ are shown in Figure~\ref{f:crttch}.

\begin{figure}[ht]
    \small
\begin{tabular}{|l|l|l|}
    \hline
$m_4 + m_3 \equiv 0 \mod 2$ & $m_6 + m_5 \equiv 0 \mod 2$ & $m_8 + m_7 \equiv 0 \mod 2$ \\
                         & $m_6+m_5+m_4 \equiv 0 \mod 3$ & $m_8+m_7+m_6+m_5 \equiv 0 \mod 4$ \\
    \hline
$m_9+m_8+m_7 \equiv 0 \mod 3$ & $m_{10}+m_9 \equiv 0 \mod 2$ & $m_{12}+m_{11} \equiv 0 \mod 2$ \\
& $m_{10}+m_9+m_8+m_7+m_6 \equiv 0 \mod 5$ & $m_{12}+m_{11}+m_{10} \equiv 0 \mod 3$  \\
& & $m_{12}+m_{11}+m_{10}+m_9 \equiv 0 \mod 4$ \\
    \hline
\end{tabular}
\caption{Consistency conditions on Tchoukaillon boards for $k \leq 12$}\label{f:crttch}
\end{figure}

\begin{proof}
Let us say that a sequence $(m_2, \ldots, m_k)$ is {\em allowable} if it
satisfies all of the congruences given in the statement of the theorem.
We will say that it is {\em realizable} if there exists some $n$ such that
$\b_i(n) = m_i$ for all $1 \leq i \leq k$.

To prove the result, we will work by induction on $k$.  If $k \leq 3$ then there
are no conditions on the $m_i$ and so every collection $(m_2, m_3)$ is
allowable.  Also, every possible combination of values with $m_2 \in \{0, 1\}$
and $m_3 \in \{0, 1, 2\}$ is actually realized by $\b(n)$ for some $n < 6$, so
the result holds.

Now suppose that the result holds for $k$.  In order to exploit
Corollary~\ref{c:periodicity}, we partition the integer interval $\{0, \ldots,
\lcm(2, \ldots, k+1)-1\}$ into $L$ subintervals, each of length $\lcm(2,
\ldots, k)$.  Hence, $L = \frac{\lcm(2, \ldots, k+1)}{\lcm(2, \ldots, k)}$.

By the induction hypothesis, we know that every allowable sequence of values $(m_1,
\ldots, m_{k})$ is realized by some board $\b(n)$ where $0 \leq n < L$.
Moreover, if we fix $m_1, \ldots, m_{k}$, then there are $L$ distinct values
for $m_{k+1}$ that are realized among the $\b_{k+1}(n)$ for $0 \leq n < \lcm(2, \ldots,
k+1)$, by the periodicity proved in Corollary~\ref{c:periodicity}.

Moreover, we claim that any realizable value for $m_{k+1}$ is allowable.  To see
this, we begin with
\[ \sum_{2 \leq j \leq k+1} m_j(n) \equiv n \mod (k+1) \]
from Theorem~\ref{t:nonit}.
Let $d$ be a divisor of $(k+1)$, so $(k+1) = m d$ for some $m \geq 2$.  
Then we can rearrange and reduce mod $d$ obtaining
\[ \sum_{d < j \leq k+1} m_j(n) \equiv n - \sum_{2 \leq j \leq d} m_j(n) \mod d \]
and the right side is zero by Theorem~\ref{t:nonit}.  If $m = 2$ we are done,
while if $m > 2$ then $d$ divides $(m-1)d$, so the inductive hypothesis gives
\[ \sum_{d < j \leq (m-1)d} m_j(n) \equiv 0 \mod d. \]
Hence, we obtain
\[ \sum_{(k+1)-d+1 < j \leq k+1} m_j(n) \equiv n - \sum_{2 \leq j \leq d} m_j(n) - \sum_{d < j \leq (m-1)d} m_j(n) \equiv 0 \mod d. \]
These imply the allowable conditions on $m_{k+1}$ given $m_2, \ldots, m_k$.

Therefore, the set of realizable values for $m_{k+1}$ are a subset of the set of
allowable values for $m_{k+1}$.  We will show that these sets are actually equal
by proving that they have the same size.  We have already shown that the set of
realizable values for $m_{k+1}$ has size $L$.  Consider the following cases:

{\bf Case:}  $(k+1)$ is a prime.  Here, we have no new allowable conditions so
all $(k+1)$ values are allowable, and this agrees with $L = (k+1)$.

{\bf Case:}  $(k+1)$ is a prime power, say $p^r$.  Here, the allowable
conditions determine $m_{k+1}$ mod $p^{r-1}$, so there are $p$ allowable values
for $m_{k+1}$.  This agrees with $L = p$.

{\bf Case:}  $(k+1)$ is a composite number.  Here, the allowable conditions
determine $m_{k+1}$ mod each maximal prime power divisor of $(k+1)$, so
$m_{k+1}$ is completely determined by the Chinese Remainder Theorem.  This
agrees with $L = 1$.

This exhausts the cases, so we obtain the result by induction.
\end{proof}

Given a sequence $(m_{i_1}, m_{i_2}, \ldots, m_{i_k})$ of integers with
$0 \leq m_{i_j} < i_j$ for all $j$, we can use Theorem~\ref{t:tcrt} to find a winning
Tchoukaillon board that agrees with the sequence, when possible.  To do this, we
perform the following steps:
\begin{enumerate}
\item[(1)]  Fix the $m_i$ that are specified by the sequence, and then try to
    find an assignment for the remaining $m_i$, where $2 \leq i < i_k$ and $i \notin \{i_1,
    i_2, \ldots, i_k\}$, such that $(m_2, m_3, \ldots, m_{i_k})$ satisfy the
    conditions given in Theorem~\ref{t:tcrt}.
\item[(2)]  Once we have assigned $m_2, m_3, \ldots, m_{i_k}$, we can
    determine the partial sums $c_2 = m_2, c_3 = m_2+m_3, \ldots, c_{i_k} =
    \sum_{j=2}^{i_k} m_j$.
\item[(3)]  Apply the Chinese Remainder Theorem~\ref{t:crt} to solve $c(n) =
    (c_2, c_3, \ldots, c_{i_k})$ for $n$.  Theorem~\ref{t:ctob} guarantees that
    a solution exists, and this $n$ gives a winning Tchoukaillon board that
    agrees with $(m_2, m_3, \ldots, m_{i_k})$.
\end{enumerate}

\begin{remark}
It does not seem to be straightforward to determine whether Step (1) in the
above algorithm can be completed for a given sequence or not.  For example, if
$(m_6, m_7, m_8, m_{10}) = (0, 1, 1, 0)$ then $m_9$ must be equivalent to $0
\mod 2$, $1 \mod 3$ and $3 \mod 5$, which implies $m_9 = 28 + 30r$ for some $r
\in \mathbb{Z}$.  Hence, $m_9$ is not realizable as a Tchoukaillon board because
$m_9 < 9$.  On the other hand, $(m_6, m_7, m_8, m_{10}) = (0, 1, 1, 1)$ is
realizable with $m_9 = 7$.

It would also be interesting to determine how to complete Step (1) in a way
which guarantees minimality of the resulting $n$.
\end{remark}

In some cases, we can give a simpler algorithm for board reconstruction.

\begin{corollary}\label{c:primerec}
Fix a sequence $(m_{i_1}, m_{i_2}, \ldots, m_{i_k})$ of integers where each $0
\leq m_{i_j} < i_j$ and each $i_j$ is a prime number.  Then the sequence agrees
with a winning Tchoukaillon board.
\end{corollary}
\begin{proof}
Using Equation~(\ref{e:tcrt}), we set the $m_i$ where $i < i_k$ and $i$ is not
prime to be $m_i = -\sum_{i-d < j < i} m_j \mod d$.  The prime indices have no
nontrivial proper divisors, so impose no conditions of the form (\ref{e:tcrt}).
Theorem~\ref{t:tcrt} then implies that there exists a winning Tchoukaillon board
$\b(n)$.  To find $n$ explicitly, construct the partial sums $\{c_i :=
\sum_{j=2}^i m_j\}_{i=2}^{i_k}$; these determine $n$ by the Chinese Remainder
Theorem~\ref{t:crt}.
\end{proof}

\begin{example}
Suppose we would like to construct a board with $(m_3, m_7) = (1, 2)$.
Then we extend using (\ref{e:tcrt}) to
\[ (m_2, m_3, \ldots, m_7) = (0, {\bf 1}, 1, 0, 2, {\bf 2}). \]
This yields the partial sums $(c_2, c_3, \ldots, c_7) = (0, 1, 2, 2, 4, 6)$
which we view as encoding a set of simultaneous congruences that are guaranteed
to be consistent by Theorems~\ref{t:tcrt} and \ref{t:ctob}, and apply the
Chinese Remainder Theorem to find agreement with $c(202)$.  Taking successive
differences of the associated $\widetilde{c}(202)$ yields the winning
Tchoukaillon board
\[ \b(202) = (0, {\bf 1}, 1, 0, 2, {\bf 2}, 4, 3, 9, 4, 8, 12, 2, 4, 6, 8, 10, 12, 14, 16, 18, 20, 22, 24). \]
It happens that there is a smaller board satisfying the constraints, namely
\[ \b(34) = (0, {\bf 1}, 1, 2, 0, {\bf 2}, 4, 6, 8, 10). \]
\end{example}

\bigskip

\begin{example}
Suppose we would like to construct a board with $(m_3, m_7, m_9) = (1, 2, 3)$.
Then (\ref{e:tcrt}) forces $m_8$ to be even and equivalent to $1 \mod 3$, so
$m_8 = 4$.  We can choose $m_2$ arbitrarily, and assign the rest of the bins by
$m_i = - \sum_{i-d < j < i} m_j$ since this happens to satisfy $m_8+m_7+m_6+m_5=0$.
Hence, we extend to
\[ (m_2, m_3, \ldots, m_9) = (0, {\bf 1}, 1, 0, 2, {\bf 2}, 4, {\bf 3}), \]
so
\[ (c_2, c_3, \ldots, c_9) = (0, 1, 2, 2, 4, 6, 2, 4). \]
This agrees with $c(202)$, yielding the same winning Tchoukaillon board as
above.  It turns out that this board is actually the smallest one among those
agreeing with $(m_3, m_7, m_9) = (1, 2, 3)$.
\end{example}

\bigskip
\section{Sowing graphs and future directions}\label{s:conclusions}

Although Tchoukaillon is a solitaire game, it is relevant for the study of many
two-player Mancala variants.  As Donkers et al. explain:
\begin{quote}
``In any mancala game that includes the rule that a player can move again if a
sowing ends in [their] own store, these [Tchoukaillon] positions are important.
These games include Kalah, Dakon, Ruma Tchuka and many others. If a Tchoukaillon
position occurs at the player's side, the player is thus able to capture all the
counters in this position\ldots \ Also mancala games that use the 2-3 capture
rule and have no stores (like Wari and Awale) benefit from Tchoukaillon
positions.''
\end{quote}
For example, Broline and Loeb point out that certain endgame positions in Ayo
are in bijection with Tchoukaillon boards \cite{Broline}; see \cite{Donkers}
for connections to other games.

Our results on partial board reconstruction are natural from this strategic
viewpoint.  While the more involved computations may be difficult for a human to
carry out during play, we imagine that computer agents for some of the
two-player Mancala variants could use these techniques to help set up ``sweep''
moves, given that some bins must remain fixed to cover other tactical goals.
These results also give new heuristics to evaluate potential moves, for humans
and computers alike.

As we have seen, Tchoukaillon also has a rich mathematical structure.
Generalizing this solitaire game to other board shapes seems to be a natural
avenue for future research.

\begin{definition}
Let $(V, E)$ be a graph with vertices $V$ and directed edges $E$.  Suppose that
$R \subset V$ is a collection of vertices that we call {\em Ruma nodes}.  We
call $S = (V, E, R)$ a {\em sowing graph}.

A labeling of the vertices of $S$ by nonnegative integers is called a {\em
Tchoukaillon board of shape $S$}.  If $b$ denotes a board of shape $S$, then we
let $b_v$ denote the label of vertex $v$.

Given a board $b$ of shape $S$, we can peform a {\em sowing move} to obtain a
new board $b'$.  To do this, we choose a vertex $v \notin R$, a Ruma node $r \in
R$, and a path in $S$ from $v$ to $r$ of edgelength $b_v$.  We then obtain $b'$
from $b$ by adding one to each vertex label along the path and assigning the
label of $v$ to be zero.

A board is said to be {\em winning} if it is possible to achieve a labeling in
which all of the non-Ruma nodes are labeled zero, using sowing moves.

We define the {\em game graph} of $S$ to be the graph whose vertices are
Tchoukaillon boards of shape $S$, in which two boards are edge-connected by a
sowing move.  Paths in the game graph from a board $b$ to the zero board
describe how to win the particular board $b$.
\end{definition}

When the sowing graph is a directed path with a single Ruma node at the sink, we
recover the game of Tchoukaillon that we have been analyzing.  When we play the
generalized game, we must still pick up all the stones from a non-Ruma node, and
sow along a path in which the last stone ends in a Ruma.  However, there may be
multiple choices for the Ruma and path from a given starting vertex.  We have
not imposed requirements on the choice of Ruma nodes in our definition, but
notice for example that any non-Ruma sinks must necessarily be labeled zero in
every winning game.

Many of the questions that we have discussed for Tchoukaillon are relevant in
this setting.  The general problem is to characterize the winning boards of a
given sowing graph and describe algorithms to win them in a minimal number of
sowing moves.  We initiate this study by characterizing the sowing graphs that
have finite game graphs.

\begin{theorem}\label{t:ifc}
A finite sowing graph has finitely many winning boards if and only if it has
no directed cycles containing both a Ruma vertex and a non-Ruma vertex.
\end{theorem}
\begin{proof}
Observe that the unplay algorithm introduced in Section~\ref{s:init} still
generates the game graph in our generalized setup.  Namely, we choose any path
from a vertex $v$ labeled zero to a Ruma node $r$, and unplay by decreasing all
of the labels along the path and setting $b_v$ to be the edgelength between $v$
and $r$.  Doing this in all possible ways from the empty board yields all
possible winning boards.

If there exists $v \notin R$ and $r \in R$ both contained in a directed cycle
then an unplay move from $r$ will always succeed, as follows.  It may be the
case that all of the non-Ruma vertices have positive labels at the start of the
unmove, but we will reduce all of the vertex labels by one each time we wrap
around the cycle.  Therefore, we can always unplay from $r$ into the closest
vertex that is labeled minimally among the labels of vertices on the cycle.
Hence, the game graph will be infinite in this case.

On the other hand, if $v$ is a vertex that is not contained in a directed cycle
with a Ruma node, then $b_v$ must be weakly less than the edgelength between $v$
and the furthest Ruma node reachable from $v$; otherwise, we could never sow
from $v$.  If this is the case for all vertices of $S \setminus R$, then we
obtain an upper bound on the number of winning boards, so the game graph will be
finite.
\end{proof}

\begin{example}
Suppose $S$ is a star graph with $k$ spokes, each of length $\l$, having a
single Ruma node at the center of the star.  Then, each of the $k$ spokes is an
independent Tchoukaillon game, so moves on different spokes commute.

Therefore, the game graph is the $k$-fold Cartesian product of the path on
$\nf(\l+1)$ vertices.  If we let $\l \rightarrow \infty$ so that each spoke is an
infinite Tchoukaillon board, as we have considered in this paper, then the game
graph becomes ``square-grid'' lattice graph $\mathbb{Z}^k$.
\end{example}

\begin{example}
Suppose $S$ is a finite directed cycle with one of the vertices chosen to be a
Ruma node.  This game is a mix of two existing solitaire Mancala games:  It can
be viewed as a variant of Tchoukaillon that allows sowing to wrap around the
board; alternatively, it can be viewed as a variant of Tchuka Ruma in which
``chaining'' moves are not allowed (see \cite{campbell-chavey}, for example).

Applying the unplay algorithm, we look for reverse paths from the Ruma.  As in
the proof of Theorem~\ref{t:ifc}, we can always unplay from the Ruma into the
closest vertex having the minimum label.  In particular, the game tree is still
a path.  Unlike Tchoukaillon, however, not every integer is attained as the
total number of stones on some board.  The first several boards are illustrated
in Figure~\ref{f:cb}.  It is an open question to characterize the integers that
do arise in this game.  
\end{example}

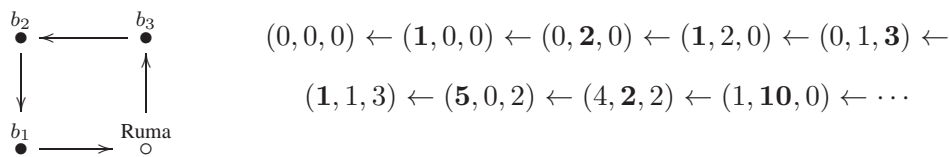
\begin{figure}[t]
\[
\xymatrix@C=2.5em@R=2em{
\stackrel{b_2}{\bullet} \ar@{->}[d] & \stackrel{b_3}{\bullet} \ar@{->}[l] \\
\stackrel{b_1}{\bullet} \ar@{->}[r] & \stackrel{\text{Ruma}}{\circ} \ar@{->}[u] \\
} 
\hspace{0.4in} 
\xymatrix@C=2.5em@R=0.5em{
(0, 0, 0) \leftarrow
({\bf 1}, 0, 0) \leftarrow
(0, {\bf 2}, 0) \leftarrow
({\bf 1}, 2, 0) \leftarrow
(0, 1, {\bf 3}) \leftarrow \\
({\bf 1}, 1, 3) \leftarrow
({\bf 5}, 0, 2) \leftarrow
(4, {\bf 2}, 2) \leftarrow
(1, {\bf 10}, 0) \leftarrow
\cdots \\
}
\]
\caption{An initial segment of the game graph for Tchoukaillon on the $4$-cycle.
Boards are labeled as $(b_1, b_2, b_3)$.}\label{f:cb}
\end{figure}

Some other questions that can be asked for any sowing graph include:
\begin{enumerate}
\item[(1)] Can we describe a general ``play'' algorithm that always solves a board in
    a minimal number of steps?  This algorithm will need to handle forks and
    cycles in the sowing graph deterministically.
\item[(2)] A finite sowing graph without cycles will have a finite number of
    winning boards.  Can we obtain a formula to count these in terms of
    properties of the graph?  What sequences are obtained for various families
    of graphs?
\item[(3)] How do the partial board reconstruction results generalize for
    sowing graphs?
\end{enumerate}

\bigskip
\section*{Acknowledgements}

Sowing games have been studied in a series of undergraduate research projects
for the last several years at James Madison University.  These projects were
partially supported by MAA-NREUP, NSF-DMS-0552763, NSF-DMS-0845277,
NSF-DMS-1004516, NSA-H98230-06-1-0156, and NSA-H98230-11-1-0215.   The authors
would like to thank the following undergraduates, whose work on sowing games
motivated the subject of this paper: Jeff Anway, Elf Bauserman, David Creech,
Brittany Dyson, Amanda Fernandez, Reginald Ford, Tyesha Hall, Mikias Kidane,
Durrell Lewis, Fierra Mason, David Melendez, Juan Carlos Ortega, Zurisadai
Pena, Spencer Sims, Melinda Vergara, Benjamin Warren, and Fanya Wyrick-Flax.


\end{document}